\documentclass[amstex,wrapfig,epsf,11pt]{amsart}

\usepackage{amssymb}
\usepackage{amsmath,amssymb,epsf,wrapfig,multicol,epic,eepic,trig,mathrsfs,dsfont,color,graphicx}
 \usepackage[abs]{overpic}
 \usepackage{here}
\usepackage{mediabb}

\theoremstyle{plain}
\newtheorem{thm}{Theorem}

\newtheorem{lem}[thm]{Lemma}

\theoremstyle{remark}

\def\hsymbu#1{\smash{\lower1.7ex\hbox{\huge$#1$}}}

\def\rmoveio#1#2{
\setlength{\unitlength}{#1}
\begin{picture}(50,30)
\put(5,0){\line(0,1){30}}

{\allinethickness{.8pt}
\put(10,15){\vector(1,0){13}}
\put(23,15){\vector(-1,0){13}}}

\qbezier(25,0)(25,20)(40,20)
\qbezier(40,20)(45,20)(45,15)
\qbezier(45,15)(45,10)(40,10)
\qbezier(40,10)(35,10)(31,14)
\qbezier(28,17)(25,25)(25,30)

\ifnum#2=2
\put(3,28){\path(0,0)(2,2)(4,0)}
\put(28,28){\path(0,0)(2,2)(4,0)}
\put(38,15){\makebox{${\Huge c_{1}}$}}
\fi

\end{picture}
}

\def\rmoveiio#1#2{
\setlength{\unitlength}{#1}
\begin{picture}(60,30)
\put(5,0){\line(0,1){30}}
\put(15,0){\line(0,1){30}}

{\allinethickness{.8pt}
\put(20,15){\vector(1,0){15}}
\put(35,15){\vector(-1,0){15}}}

\qbezier(40,0)(42,1)(47,3)
\qbezier(52,6)(68,15)(52,24)
\qbezier(47,27)(42,30)(40,30)

\qbezier(60,0)(20,15)(60,30)

\ifnum#2=2
\put(2,27){\path(0,0)(3,3)(6,0)}
\put(12,27){\path(0,0)(3,3)(6,0)}
\put(40,27){\path(0,0)(0,3)(3,3)}
\put(60,27){\path(0,0)(0,3)(-3,3)}
\put(47,16){\makebox{${\Huge c_{1}}$}}
\put(47,10){\makebox{${\Huge c_{2}}$}}
\fi

\ifnum#2=3
\put(2,2){\path(0,0)(3,-3)(6,0)}
\put(12,27){\path(0,0)(3,3)(6,0)}
\put(40,3){\path(0,0)(0,-3)(3,-3)}
\put(60,27){\path(0,0)(0,3)(-3,3)}
\put(47,16){\makebox{${\Huge c_{1}}$}}
\put(47,10){\makebox{${\Huge c_{2}}$}}
\fi

\end{picture}
}
\def\rmoveiiio#1#2{
\setlength{\unitlength}{#1}
\begin{picture}(75,30)
\put(0,0){\line(1,1){15}}
\qbezier(15,15)(20,20)(20,30)

\put(10,0){\line(-1,1){4}}
\qbezier(4,6)(-5,15)(5,25)
\put(5,25){\line(1,1){5}}

\qbezier(20,0)(20,10)(16,14)
\put(14,16){\line(-1,1){8}}
\put(4,26){\line(-1,1){4}}

{\allinethickness{.8pt}
\put(25,15){\vector(1,0){15}}
\put(40,15){\vector(-1,0){15}}}

\qbezier(50,0)(50,10)(55,15)
\put(55,15){\line(1,1){15}}

\put(60,0){\line(1,1){5}}
\qbezier(65,5)(75,15)(66,24)
\put(64,26){\line(-1,1){4}}

\put(70,0){\line(-1,1){4}}
\put(64,6){\line(-1,1){8}}
\qbezier(54,16)(50,20)(50,30)

\ifnum#2=2
\put(0,27){\path(0,0)(0,3)(3,3)}
\put(7,30){\path(0,0)(3,0)(3,-3)}
\put(17,27){\path(0,0)(3,3)(6,0)}

\put(10,23){\makebox{${\Huge c_{1}}$}}
\put(10,3){\makebox{${\Huge c_{2}}$}}
\put(20,13){\makebox{${\Huge c_{3}}$}}

\put(47,27){\path(0,0)(3,3)(6,0)}
\put(60,27){\path(0,0)(0,3)(3,3)}
\put(67,30){\path(0,0)(3,0)(3,-3)}

\put(70,23){\makebox{${\Huge c'_{2}}$}}
\put(70,3){\makebox{${\Huge c'_{1}}$}}
\put(60,13){\makebox{${\Huge c'_{3}}$}}
\fi

\end{picture}
}
\def\rmovevio#1#2{
\setlength{\unitlength}{#1}
\begin{picture}(50,30)
\put(5,0){\line(0,1){30}}

{\allinethickness{.8pt}
\put(10,15){\vector(1,0){15}}
\put(25,15){\vector(-1,0){15}}}

\qbezier(30,0)(30,20)(45,20)
\qbezier(45,20)(50,20)(50,15)
\qbezier(50,15)(50,10)(45,10)
\qbezier(45,10)(30,10)(30,30)

\put(34,15){\circle{5}}

\ifnum#2=2
\put(3,28){\path(0,0)(2,2)(4,0)}
\put(28,28){\path(0,0)(2,2)(4,0)}
\put(38,15){\makebox{${\Huge c_{1}}$}}
\fi

\end{picture}
}

\def\rmoveviio#1#2{
\setlength{\unitlength}{#1}
\begin{picture}(60,30)
\put(5,0){\line(0,1){30}}
\put(15,0){\line(0,1){30}}

{\allinethickness{.8pt}
\put(20,15){\vector(1,0){15}}
\put(35,15){\vector(-1,0){15}}}

\qbezier(40,0)(80,15)(40,30)
\qbezier(60,0)(20,15)(60,30)
\put(50,4){\circle{5}}
\put(50,25){\circle{5}}

\ifnum#2=2
\put(2,27){\path(0,0)(3,3)(6,0)}
\put(12,27){\path(0,0)(3,3)(6,0)}
\put(40,27){\path(0,0)(0,3)(3,3)}
\put(60,27){\path(0,0)(0,3)(-3,3)}
\put(47,15){\makebox{${\Huge c_{1}}$}}
\put(47,10){\makebox{${\Huge c_{2}}$}}
\fi

\ifnum#2=3
\put(2,2){\path(0,0)(3,-3)(6,0)}
\put(12,27){\path(0,0)(3,3)(6,0)}
\put(40,3){\path(0,0)(0,-3)(3,-3)}
\put(60,27){\path(0,0)(0,3)(-3,3)}
\put(47,15){\makebox{${\Huge c_{1}}$}}
\put(47,10){\makebox{${\Huge c_{2}}$}}
\fi

\end{picture}
}

\def\rmoveviiio#1#2{
\setlength{\unitlength}{#1}
\begin{picture}(70,30)
\put(0,0){\line(1,1){15}}
\qbezier(15,15)(20,20)(20,30)

\put(10,0){\line(-1,1){5}}
\qbezier(5,5)(-5,15)(5,25)
\put(5,25){\line(1,1){5}}

\qbezier(20,0)(20,10)(15,15)
\put(15,15){\line(-1,1){15}}

\put(5,5){\circle{5}}
\put(15,15){\circle{5}}
\put(5,25){\circle{5}}

{\allinethickness{.8pt}
\put(28,15){\vector(1,0){14}}
\put(42,15){\vector(-1,0){14}}}

\qbezier(50,0)(50,10)(55,15)
\put(55,15){\line(1,1){15}}

\put(60,0){\line(1,1){5}}
\qbezier(65,5)(75,15)(65,25)
\put(65,25){\line(-1,1){5}}

\put(70,0){\line(-1,1){15}}
\qbezier(55,15)(50,20)(50,30)

\put(65,5){\circle{5}}
\put(55,15){\circle{5}}
\put(65,25){\circle{5}}

\ifnum#2=2
\put(0,27){\path(0,0)(0,3)(3,3)}
\put(7,30){\path(0,0)(3,0)(3,-3)}
\put(17,27){\path(0,0)(3,3)(6,0)}

\put(20,13){\makebox{${\Huge c_{1}}$}}

\put(47,27){\path(0,0)(3,3)(6,0)}
\put(60,27){\path(0,0)(0,3)(3,3)}
\put(67,30){\path(0,0)(3,0)(3,-3)}

\put(60,13){\makebox{${\Huge c'_{1}}$}}
\fi

\end{picture}
}

\def\rmovevivo#1#2{
\setlength{\unitlength}{#1}
\begin{picture}(70,30)
\put(0,0){\line(1,1){15}}
\qbezier(15,15)(20,20)(20,30)

\put(10,0){\line(-1,1){4}}
\qbezier(4,6)(-5,15)(5,25)
\put(5,25){\line(1,1){5}}

\qbezier(20,0)(20,10)(16,14)
\put(16,14){\line(-1,1){16}}

\put(15,15){\circle{5}}
\put(5,25){\circle{5}}

{\allinethickness{.8pt}
\put(28,15){\vector(1,0){14}}
\put(42,15){\vector(-1,0){14}}}

\qbezier(50,0)(50,10)(55,15)
\put(55,15){\line(1,1){15}}

\put(60,0){\line(1,1){5}}
\qbezier(65,5)(75,15)(67,25)
\put(65,27){\line(-1,1){3}}

\put(70,0){\line(-1,1){16}}
\qbezier(54,16)(50,20)(50,30)

\put(65,5){\circle{5}}
\put(55,15){\circle{5}}

\ifnum#2=2
\put(0,27){\path(0,0)(0,3)(3,3)}
\put(7,30){\path(0,0)(3,0)(3,-3)}
\put(17,27){\path(0,0)(3,3)(6,0)}

\put(20,13){\makebox{${\Huge c_{1}}$}}

\put(47,27){\path(0,0)(3,3)(6,0)}
\put(60,27){\path(0,0)(0,3)(3,3)}
\put(67,30){\path(0,0)(3,0)(3,-3)}

\put(60,13){\makebox{${\Huge c'_{1}}$}}
\fi

\end{picture}
}
\def\rmovevfo#1#2{
\setlength{\unitlength}{#1}
\begin{picture}(70,30)
\put(0,0){\line(1,1){15}}
\qbezier(15,15)(20,20)(20,30)

\put(10,0){\line(-1,1){4}}
\qbezier(4,6)(-5,15)(5,25)
\put(5,25){\line(1,1){5}}

\qbezier(20,0)(20,10)(16,14)
\put(14,16){\line(-1,1){14}}

\put(5,25){\circle{5}}

{\allinethickness{.8pt}
\put(28,15){\vector(1,0){14}}
\put(42,15){\vector(-1,0){14}}}

\qbezier(50,0)(50,10)(55,15)
\put(55,15){\line(1,1){15}}

\put(60,0){\line(1,1){5}}
\qbezier(65,5)(75,15)(66,24)
\put(64,26){\line(-1,1){4}}

\put(70,0){\line(-1,1){14}}
\qbezier(54,16)(50,20)(50,30)

\put(65,5){\circle{5}}

\ifnum#2=2
\put(0,27){\path(0,0)(0,3)(3,3)}
\put(7,30){\path(0,0)(3,0)(3,-3)}
\put(17,27){\path(0,0)(3,3)(6,0)}

\put(20,13){\makebox{${\Huge c_{1}}$}}

\put(47,27){\path(0,0)(3,3)(6,0)}
\put(60,27){\path(0,0)(0,3)(3,3)}
\put(67,30){\path(0,0)(3,0)(3,-3)}

\put(60,13){\makebox{${\Huge c'_{1}}$}}
\fi

\end{picture}
}

\def\canocutsysic#1#2{
\setlength{\unitlength}{#1}
\begin{picture}(130,20)

\put(5,0){\line(1,1){20}}
\put(25,0){\line(-1,1){8}}
\put(5,20){\line(1,-1){8}}

\put(20,15){\path(0,0)(3,0)(0,3)(0,0)}
\put(20,5){\path(0,0)(3,0)(0,-3)(0,0)}

\put(5,0){\path(0,3)(0,0)(3,0)}
\put(25,0){\path(-3,0)(0,0)(0,3)}

\put(50,0){\line(1,1){8}}
\put(70,20){\line(-1,-1){8}}
\put(70,0){\line(-1,1){20}}

\put(65,15){\path(0,0)(3,0)(0,3)(0,0)}
\put(65,5){\path(0,0)(3,0)(0,-3)(0,0)}

\put(50,0){\path(0,3)(0,0)(3,0)}
\put(70,0){\path(-3,0)(0,0)(0,3)}

\ifnum#2=2
\put(0,15){\makebox{{\small $0$}}}
\put(0,0){\makebox{{\small $0$}}}
\put(27,15){\makebox{{\small $0$}}}
\put(27,0){\makebox{{\small $0$}}}
\put(15,14){\makebox{{\small $1$}}}
\put(15,1){\makebox{{\small $1$}}}
\put(45,15){\makebox{{\small $0$}}}
\put(45,0){\makebox{{\small $0$}}}
\put(72,15){\makebox{{\small $0$}}}
\put(72,0){\makebox{{\small $0$}}}
\put(60,14){\makebox{{\small $1$}}}
\put(60,1){\makebox{{\small $1$}}}

\fi

\put(95,0){\line(1,1){20}}
\put(115,0){\line(-1,1){20}}
\put(105,10){\circle{4}}

\put(95,0){\path(0,3)(0,0)(3,0)}
\put(115,0){\path(-3,0)(0,0)(0,3)}

\ifnum#2=2
\put(90,15){\makebox{{\small $0$}}}
\put(90,0){\makebox{{\small $0$}}}
\put(117,15){\makebox{{\small $0$}}}
\put(117,0){\makebox{{\small $0$}}}
\fi

\end{picture}
}

\begin{document}
\title[Diagrams realizing prescribed sublink diagrams]
{Diagrams realizing prescribed sublink diagrams for virtual links and welded links}
\author{Naoko Kamada} 
\thanks{This work was supported by JSPS KAKENHI Grant Number 19K03496.}
\address{Graduate School of Science,  Nagoya City University\\ 
1 Yamanohata, Mizuho-cho, Mizuho-ku, Nagoya, Aichi 467-8501 Japan
}

\date{}

\begin{abstract} 
Jin and Lee \cite{rjin2001}  proved the following: 
Suppose that  $D_1, \dots, D_n$ are  link diagrams.  Given a link $L$ which is partitioned into sublinks  $L_1, \dots, L_n$ admitting diagrams $D_1, \dots, D_n$ respectively, there is a diagram $D$ of $L$ whose restrictions to $L_1, \dots, L_n$ are isotopic to 
$D_1, \dots, D_n$, respectively. 
In this paper we show that a similar result does hold for welded links and does not for virtual links.

\end{abstract}
\maketitle

\section{Introduction}
G. T. Jin and J. H. Lee \cite{rjin2001} proved the following theorem. 

\begin{thm}[G. T. Jin and J. H. Lee \cite{rjin2001}]\label{lj1}
Suppose that  $D_1, \dots, D_n$ are  link diagrams.  Given a link $L$ which is partitioned into sublinks  $L_1, \dots, L_n$ admitting diagrams $D_1, \dots, D_n$ respectively, there is a diagram $D$ of $L$ whose restrictions to $L_1, \dots, L_n$ are isotopic in $\mathbb{R}^2$ to 
$D_1, \dots, D_n$, respectively. 
\end{thm}

Virtual links were introduced by  L.~H.~Kauffman  \cite{rkaud} as equivalence classes of virtual link diagrams in $\mathbb{R}^2$ under a certain equivalence relation.  
They are in one-to-one correspondence with stable equivalence classes of links in thickened surfaces \cite{rcks,rkk}. 
   Welded links were introduced by R. Fenn R. Rimanyi, and C. Rouke \cite{rfrr} as equivalence classes of welded link diagrams in $\mathbb{R}^2$ under another equivalence relation.  
   Both virtual links and welded links are generalizations of links.  

In this paper we prove the following.  

\begin{thm}\label{main1}
Suppose that  $D_1, \dots, D_n$ are diagrams of welded links.  Given a welded link $L$ which is partitioned into sublinks  $L_1, \dots, L_n$ admitting diagrams $D_1, \dots, D_n$ respectively, there is a diagram $D$ of $L$ whose restrictions to $L_1, \dots, L_n$ are isotopic in $\mathbb{R}^2$ to 
$D_1, \dots, D_n$, respectively. 
\end{thm}

A similar statement to Theorems~\ref{lj1} and \ref{main1} does not hold for virtual links. 

\begin{thm}\label{main2}
There exist diagrams $D_1$ and $D_2$ of virtual links, and a virtual link $L$ which is partitioned  into $L_1$ and $L_2$ admitting diagrams $D_1$ and $D_2$ such that $L$ does not admit any diagram $D$ whose restrictions to $L_1$ and $L_2$ are isotopic in $\mathbb{R}^2$ to $D_1$ and $D_2$, respectively.
\end{thm}

\section{Virtual links and welded links}\label{sect:virtualwelded} 

We recall virtual links and welded links.  

In this paper a {\it diagram} means a collection of immersed oriented loops in $\mathbb{R}^2$ such that the multiple points are transverse double points which are classified into classical crossings and virtual crossings: 
A {\it classical crossing} is a crossing with over/under information as usual in knot theory, and a {\it virtual crossing} is a crossing without over/under information \cite{rkaud}.  A virtual crossing is depicted as a crossing encircled with a small circle. (Such a circle is not considered as a component of the diagram.)  A classical crossing is also called a positive or negative crossing according to the sign of the crossing as usual in knot theory.  

Two diagrams are  {\it v-equivalent} or {\it equivalent as virtual links}  if they are related by a finite sequence of local moves depicted in Figure~ \ref{fgmoves} except WR up to isotopies of $\mathbb{R}^2$.  A {\it virtual link} is an equivalence class of diagrams under this equivalence relation.  

Two diagrams are {\it w-equivalent} or {\it equivalent as welded links}  if they are related by a finite sequence of local moves depicted in Figure~\ref{fgmoves}  
up to isotopies of $\mathbb{R}^2$.  A {\it welded link} is an equivalence class of diagrams under this equivalence relation.  

A diagram without virtual crossings is called a {\it classical link diagram}.  
Two classical link diagrams are {\it r-equivalent} or {\it equivalent as classical links}  if they are related by a finite sequence of local moves R1, R2 and R3 depicted in Figure~\ref{fgmoves}  
up to isotopies of $\mathbb{R}^2$.  

It is known that two classical link diagrams are equivalent as classical links if and only if they are equivalent as virtual (or welded) links.  In this sense, virtual links and welded links are generalizations of classical links.

%
\begin{figure}[H]
\centerline{
\begin{tabular}{ccc}
\rmoveio{.4mm}{1}&\rmoveiio{.4mm}{1}&\rmoveiiio{.4mm}{1}\\
R1 & R2 &  R3 \\
\end{tabular}}
\vspace{0.2cm}
\centerline{
\begin{tabular}{cccc}
\rmovevio{.4mm}{1}&\rmoveviio{.4mm}{1}&\rmoveviiio{.4mm}{1}&\rmovevivo{.4mm}{1}\\
VR1 &  VR2 & VR3 & VR4 \\
\end{tabular}}
\vspace{0.2cm}
\centerline{
\begin{tabular}{c}
\rmovevfo{.4mm}{1}\\
WR \\
\end{tabular}}
\caption{Moves}\label{fgmoves}
\end{figure}

Let $D'$ be a diagram obtained from a diagram $D$ by one of the local moves 
depicted in Figure~\ref{fgmoves}, a {\it support} of the move is a 
region $M$ in $\mathbb{R}^2$ which is homeomorphic to the $2$-disk such that 
 $D$ and $D'$ are identical in $\mathbb{R}^2 \setminus M$ and that 
$D \cap M$ and $D' \cap D'$ are depicted in the figure.

A {\it detour move} is a deformation of a diagram depicted in Figure~\ref{fig:detour} (i), where the box stands for a diagram which does not change.  
Two diagrams related by a detour move are equivalent as virtual links and as welded links.  

An {\it over detour move} is a deformation of a virtual/welded link diagram in Figure~\ref{fig:detour} (ii). 
Two diagrams related by a detour move are equivalent as welded links. 

\begin{figure}[H]
\begin{center}
\includegraphics[width=13.cm]{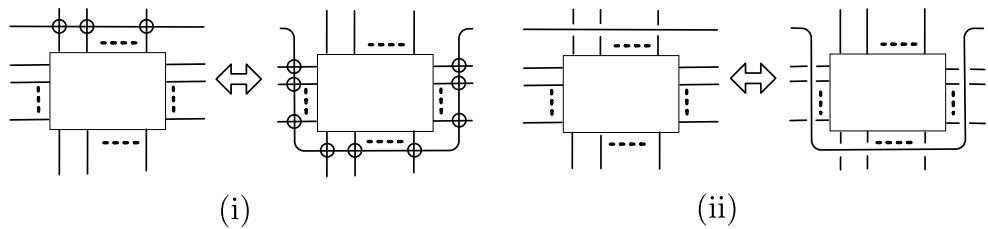}
\end{center}
\caption{Two detour moves}\label{fig:detour}
\end{figure}

\section{Proof of Theorem \ref{main1}}\label{sect:main1}

We introduce three moves depicted in 
Figure~\ref{fig:line2}, which do not change the equivalence class of a diagram as a welded link.  

For each move in the figure, in the left hand side of the move, let $\alpha$ be the vertical arc and 
let $v_0, v_1, \dots, v_m$ be crossings on $\alpha$ appearing  in this order from the bottom to the top. Let $\beta_0, \beta_1, \dots, \beta_m$ be the arcs intersecting $\alpha$ at $v_0, v_1, \dots, v_m$, respectively. Move $\beta$ toward the top along $\alpha$ and we obtain an arc $\beta_0'$ as in the right hand side.  

(i) An {\it under finger move} is as follows:  At $v_0$, $\beta_0$ is under $\alpha$.  
For each $i \in \{1, \dots, m\}$, when $v_i$ is a classical crossing, the two crossings of $\beta_0'$ and $\beta_i$ are classical crossings where $\beta_0'$ is under $\beta_i$.  When $v_i$ is a virtual crossing, the two crossings of $\beta_0'$ and $\beta_i$ are virtual crossings.  The intersection of $\beta_0'$ and $\alpha$ is a classical crossing where $\beta_0'$ is under $\alpha$.  

(ii) An {\it over finger move} is as follows: At $v_0$, $\beta_0$ is over $\alpha$.  
For each $i \in \{1, \dots, m\}$, the two crossings of $\beta_0'$ and $\beta_i$ are classical crossings where $\beta_0'$ is over $\beta_i$.   The intersection of $\beta_0'$ and $\alpha$ is a classical crossing where $\beta_0'$ is over $\alpha$.  
  
(iii) A {\it virtual finger move} is as follows:  At $v_0$, $\beta_0$ meets $\alpha$ as a virtual crossing.  
For each $i \in \{1, \dots, m\}$, the two crossings of $\beta_0'$ and $\beta_i$ are virtual crossings.   The intersection of $\beta_0'$ and $\alpha$ is also a virtual crossing.  

An over finger move  is an over detour move, and a virtual finger move is a detour move.  

\begin{figure}[ht]
\begin{center}
\includegraphics[width=10.cm]{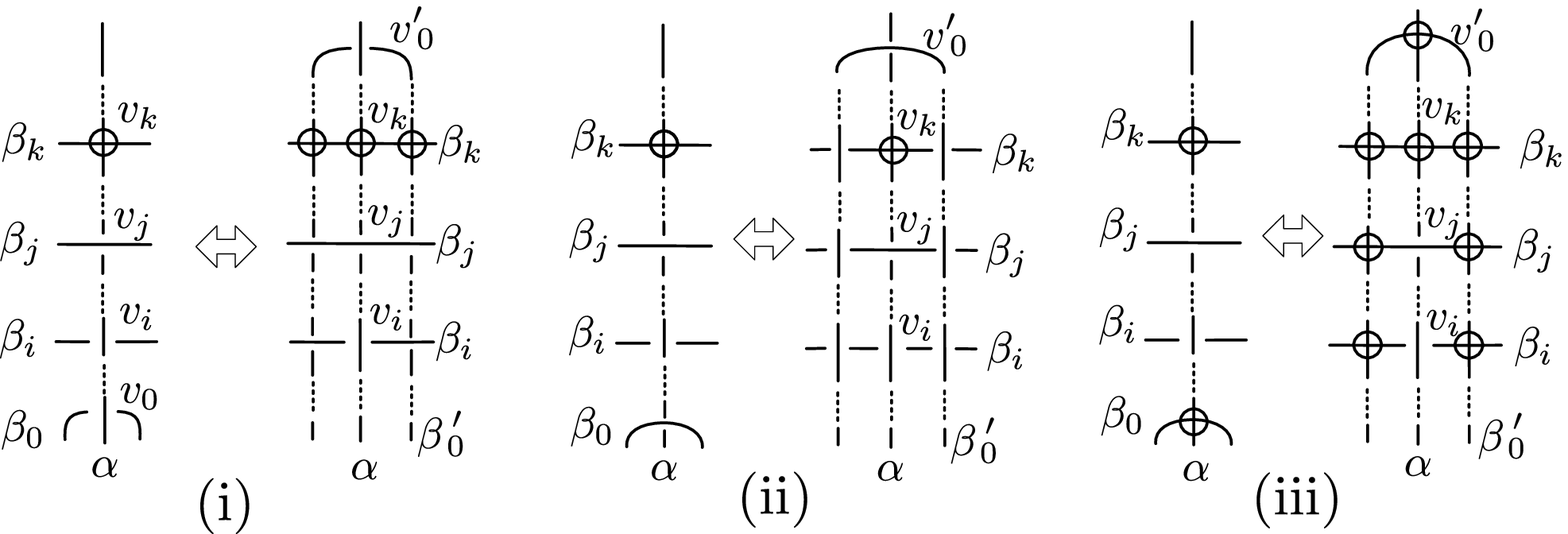}
\end{center}
\caption{Under, over and virtual finger moves}\label{fig:line2}
\end{figure}

Given a diagram of $D$ of a welded link $L$ and a sublink $L_0$, the restriction of $D$ to $L_0$ is the diagram obtained from $D$ by removing the components not belonging to $L_0$.  It is denoted by $D(L_0)$.

\begin{lem}\label{lem:A}
Let $D$ be a diagram of a welded link $L$ partitioned into $L_1$ and $L_2$.  
Let $D_1$ be a diagram obtained from $D(L_1)$ by a local move depicted in Figure~\ref{fgmoves}.  
There is a diagram $D'$ of $L$ such that $D'(L_1)$ is isotopic to $D_1$ in $\mathbb{R}^2$ and $D'(L_2) = D(L_2)$. 
\end{lem}

\begin{proof}
We say that a simple arc $\gamma$ in a diagram is an {\it arc of type (i), (ii), (iii), (iv) or (v)} 
and it is drawn with a thick line, a thick dotted line, a thin line, a thin dotted line, or a thin dashed line, respectively, 
if 
one of the following conditions (i)--(v)  is satisfied respectively, see Figure~\ref{fig:line1}:  

\begin{itemize}
\item[(i)] (On a thick line,) no condition is required.  
\item[(ii)] (On a thick  dotted line,) at each classical crossing on $\gamma$, the arc $\gamma$ is an under arc.
\item[(iii)] (On a thin line,) there is no crossing on $\gamma$.
\item[(iv)] (On a thin dotted line,) every crossing on $\gamma$ is a classical crossing where $\gamma$ is an over arc. 
\item[(v)] (On a thin dashed line,) every crossing on $\gamma$ is a virtual crossing. 
\end{itemize}

\begin{figure}[H]
\begin{center}
\includegraphics[width=10.cm]{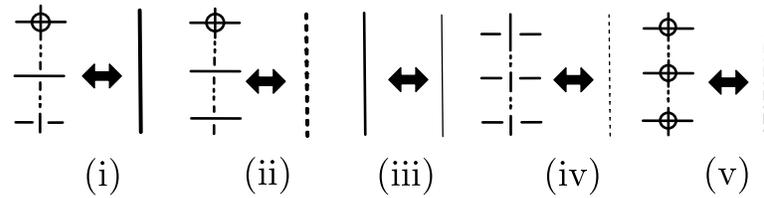}
\end{center}
\caption{segments of virtual link diagrams}\label{fig:line1}
\end{figure}

For example, the three moves in Figure~\ref{fig:line3} stand for the three finger moves 
in Figure~\ref{fig:line2}.

 
\begin{figure}[H]
\begin{center}
\includegraphics[width=10.cm]{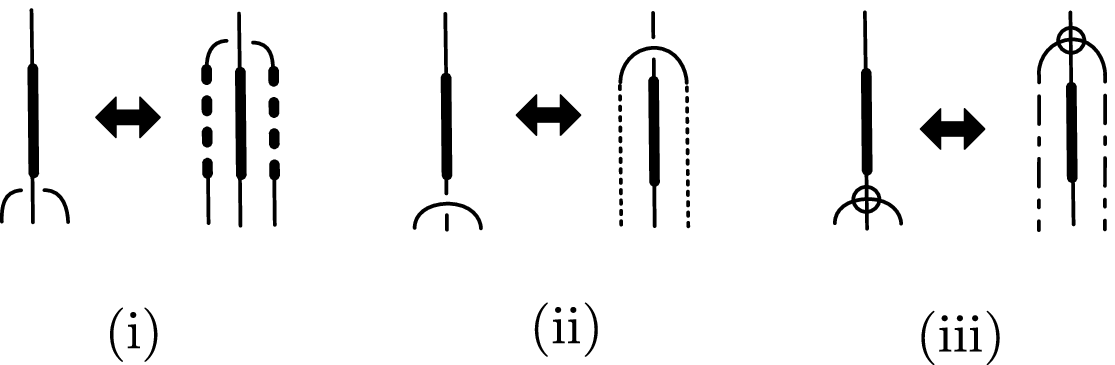}
\end{center}
\caption{Under, over and virtual finger moves}\label{fig:line3}
\end{figure}

In what follows, let $M$ be a support of the local move sending $D(L_1)$ to $D_1$. 

The case of R1. 
First we consider the case where an R1 move from left to right in Figure~\ref{fgmoves} is applied.  
Take a small arc $A$ in $D(L_1)$ in $M$ and take a small rectangular disk in $M$, say $\Delta$, containing $A$ and avoiding $D(L_2)$. 
Apply an  R1 move from the left to the right in $\Delta$ as in Figure~\ref{fig:prmove1} (i) and we obtain a desired diagram $D'$.   (The shaded region is $M \setminus \Delta$.)

We consider the case where an R1 move from right to left in Figure~\ref{fgmoves} is applied.  
Let $c$ be a self crossing of $D(L_1)$ which is removed by the R1 move.  Take a small rectangular region in $M$, say $\Delta$, containing $c$ and avoiding $D(L_2)$.  Apply a sequence of deformation to $D(L_1)$ 
as in Figure~\ref{fig:prmove1} (ii) and we obtain a desired diagram $D'$.   (The shaded region is $M \setminus \Delta$. The first deformation is an over finger move.  The second deformation is an over detour move.  
The third deformation is an R1 move.)  

\begin{figure}[ht]
\begin{center}
\includegraphics[width=12.cm]{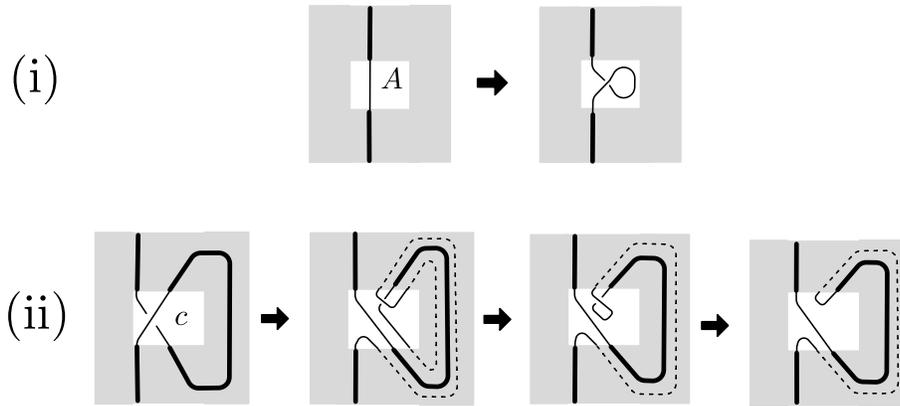}
\end{center}
\caption{The case of R1}\label{fig:prmove1}
\end{figure}

The case of VR1. 
We consider the case where a VR1 move in Figure~\ref{fgmoves} is applied.  
By a similar argument to the case of an R1 move, we have a desired diagram $D'$. 
See Figure~\ref{fig:pvrmove1}. 
(In (ii), the first deformation is a virtual finger move.  The second deformation is a detour move.  
The third deformation is a VR1 move.)  

\begin{figure}[ht]
\begin{center}
\includegraphics[width=12.cm]{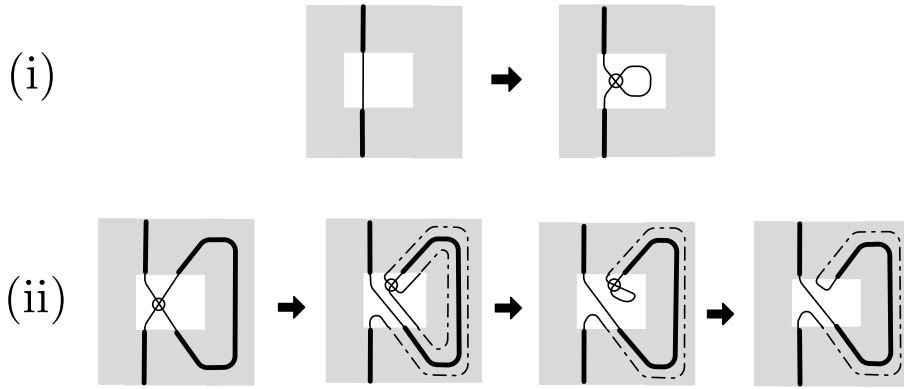}
\end{center}
\caption{The case of VR1 move}\label{fig:pvrmove1}
\end{figure}

The case of R2.  
We consider the case where an R2 move from left to right in Figure~\ref{fgmoves} is applied.  
Take small arcs $A_1, A_2$ in $D(L_1)$ and take a pair of small rectangular regions in $M$, say $\Delta_1, \Delta_2$ containing $A_1, A_2$ and avoiding $D(L_2)$ as in the left of Figure~\ref{fig:prmove2} (i).  Applying an over finger move as in Figure~\ref{fig:prmove2} (i) and we obtain a desired diagram $D'$.   (The shaded region is $M \setminus (\Delta_1 \cup \Delta_2)$.)

We consider the case where an R2 move from right to left in Figure~\ref{fgmoves} is applied.  
Let $c_1, c_2$ be  self crossings of $D(L_1)$ which are removed by the R2 move in $M$.   Take small rectangular regions in $M$, say $\Delta_1, \Delta_2$, containing $c_1, c_2$ and avoiding $D(L_2)$.  Apply a sequence of deformation to $D(L_1)$ 
as in Figure~\ref{fig:prmove2} (ii) and we obtain a desired diagram $D'$.  (The first deformation is an under finger move.  The second deformation is an over detour move.)

\begin{figure}[H]
\begin{center}
\includegraphics[width=10cm]{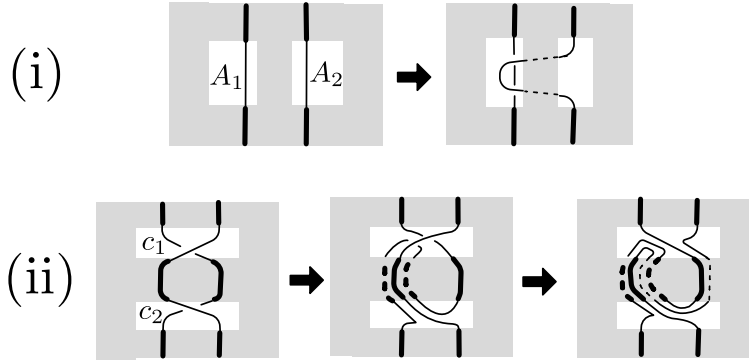}
\end{center}
\caption{The case of R2}\label{fig:prmove2}
\end{figure}

The case of VR2.  
By a similar argument to the case of R2, we have a desired diagram $D'$. 
See Figure~\ref{fig:pvrmove2}. 
(In (i), the deformation is a virtual finger move.  In (ii), the first deformation is a virtual finger move. The second deformation is a detour move.)

\begin{figure}[H]
\begin{center}
\includegraphics[width=10cm]{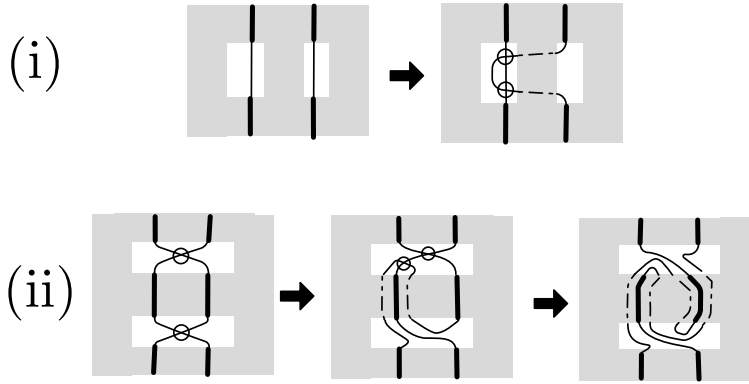}
\end{center}
\caption{The case of VR2}\label{fig:pvrmove2}
\end{figure}

The case of R3.  
We consider the case where an R3 move from left to right in Figure~\ref{fgmoves} is applied.  
Let $c_1, c_2$ and $c_3$ be the three crossings of $D(L_1)$ where the R3 move is applied, and take small rectangular regions in $M$, say say $\Delta_1, \Delta_2, \Delta_3$, containing $c_1, c_2, c_3$ and avoiding $D(L_2)$.  Apply a sequence of deformation to $D(L_1)$ 
as in Figure~\ref{fig:prmove3} and we obtain a desired diagram $D'$.  (The shaded region is $M \setminus (\Delta_1 \cup \Delta_2 \cup \Delta_3)$. 
The first deformation is a consecutive application of  
two under finger moves.  The second deformation is an over detour move.)  

An R3 move from right to left is not necessary, since it is obtained from an R3 move from left to right by rotating the figure by 180 degree.  

\begin{figure}[ht]
\begin{center}
\includegraphics[width=7.cm]{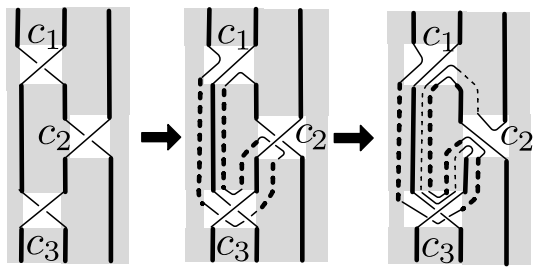}
\end{center}
\caption{The case of R3}\label{fig:prmove3}
\end{figure}

The case of VR3. 
We consider the case where a VR3 move from left to right in Figure~\ref{fgmoves} is applied.  
 Apply a sequence of deformation to $D(L_1)$ 
as in Figure~\ref{fig:pvrmove3} and we obtain a desired diagram $D'$.  
(The first deformation is a consecutive application of  
two virtual finger moves.  The second deformation is a detour move.)  
A VR3 move from right to left is not necessary.  

\begin{figure}[H]
\begin{center}
\includegraphics[width=7.cm]{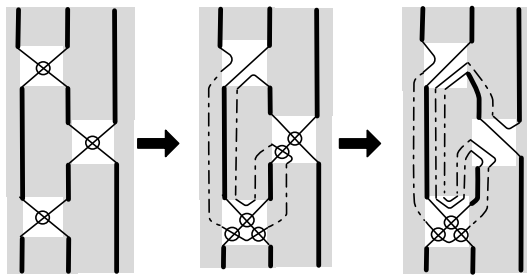}
\end{center}
\caption{The case of VR3}\label{fig:pvrmove3}
\end{figure}

The case of VR4.  
We consider the case where an VR4 move from left to right in Figure~\ref{fgmoves} is applied. 
Apply a sequence of deformation to $D(L_1)$ 
as in Figure~\ref{fig:pvrmove4} and we obtain a desired diagram $D'$.  
(The first deformation is a consecutive application of  
two  virtual finger moves.  The second deformation is a detour move.)  
A VR4 move from right to left is not necessary.  

\begin{figure}[ht]
\begin{center}
\includegraphics[width=7.cm]{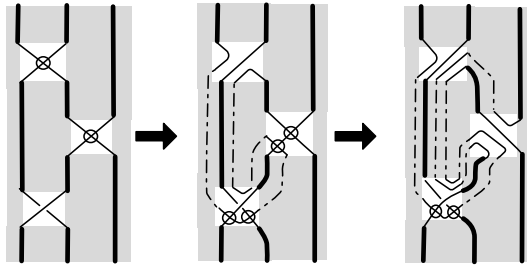}
\end{center}
\caption{The case of VR4}\label{fig:pvrmove4}
\end{figure}

The case of WR.  
We consider the case where an WR move from left to right in Figure~\ref{fgmoves} is applied. 
Apply a sequence of deformation to $D(L_1)$ 
as in Figure~\ref{fig:pfmove} and we obtain a desired diagram $D'$.  
(The first deformation is a consecutive application of  
two over finger move.  The second deformation is an under finger move. 
The third deformation is an over detour move.)  
A WR move from right to left is not necessary.  
\end{proof}

\begin{figure}[ht]
\begin{center}
\includegraphics[width=9.5cm]{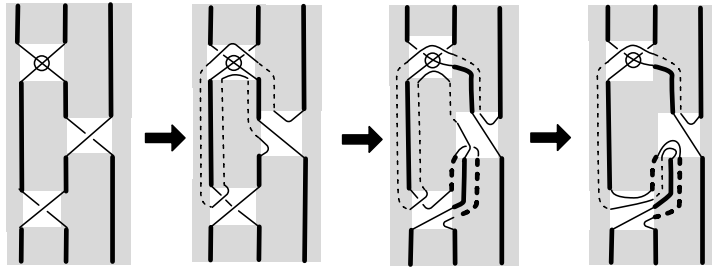}
\end{center}
\caption{The case of WR}\label{fig:pfmove}
\end{figure}

\begin{lem}\label{lem:B}
Let $D$ be a diagram of a welded link $L$ partitioned into $L_1$ and $L_2$.  
Let $D_1$ be a diagram of $L_1$.  
There is a diagram $D'$ of $L$ such that $D'(L_1)$ is isotopic to $D_1$ in $\mathbb{R}^2$ and $D'(L_2) = D(L_2)$. 
\end{lem}

\begin{proof}
Using Lemma~\ref{lem:A} inductively, we obtain the result.  
\end{proof}

\begin{proof}[Proof of Theorem~\ref{main1}]
Suppose that  $D_1, \dots, D_n$ are  welded link diagrams.  Let $L$ be a welded link partitioned into sublinks  $L_1, \dots, L_n$ admitting diagrams $D_1, \dots, D_n$ respectively.  

Let $D= D^{(0)}$ be a diagram of $L$.  
By considering $L$ to be partitioned into $L_1$ and $L \setminus L_1$ and applying Lemma~\ref{lem:B}, there is a diagram $D^{(1)}$ of $L$ such that $D^{(0)}(L_1)$ is isotopic to $D^{(1)}(L_1)$  and 
$D^{(0)}(L \setminus L_1) = D^{(1)}(L \setminus L_1)$.  

Inductively, for $i=2, \dots, n$, assume that we have a diagram $D^{(i-1)}$.  
By considering $L$ to be partitioned into $L_i$ and $L \setminus L_i$ and applying Lemma~\ref{lem:B}, there is a diagram $D^{(i)}$ of $L$ such that $D^{(i-1)}(L_i)$ is isotopic to $D^{(i)}(L_i)$  and 
$D^{(i-1)}(L \setminus L_i) = D^{(i)}(L \setminus L_i)$.  

Then $D'= D^{(n)}$ is a desired diagram.  
\end{proof}

\section{Proof of Theorem \ref{main2}}\label{sect:main2}

In order to prove Theorem~\ref{main2}, we use the notion of the 2-cyclic covering of a virtual link, introduced in \cite{rkn3,rkn4}. 

Let $D$ be a diagram. 
Moving $D$ by an isotopy of $\mathbb{R}^2$, we assume that $D$ is 
on the left of the $y$-axis and all crossings have distinct $y$-coordinates. 
Let $D^*$ be a copy of $D$ on the right of the $y$-axis which is obtained from $D$ by sliding along the $x$-axis. 
Let $v_1,\dots, v_k$ be the virtual crossings of $D$ and 
let $v^*_1,\dots, v^*_k$ be the corresponding virtual crossings of $D^*$. 
For each $i \in \{1,\dots ,k\}$, we denote by $l_i $ the horizontal line containing 
$v_i$ and $v_i^*$, and let $N(l_i)$ be a regular neighborhood of $l_i$ in $\mathbb{R}^2$.  
   Consider a diagram, denoted by $\widetilde{D}$, obtained from $D \cup  D^*$ by replacing 
the intersection with $N(l_i)$ for each $i \in \{1,\dots , k\}$ as in Figure~\ref{fig:defwcov2}. 
We call the diagram $\widetilde{D}$ a {\it 2-cyclic covering diagram} of $D$.

\begin{figure}[ht]
\begin{center}
\includegraphics[width=14.cm]{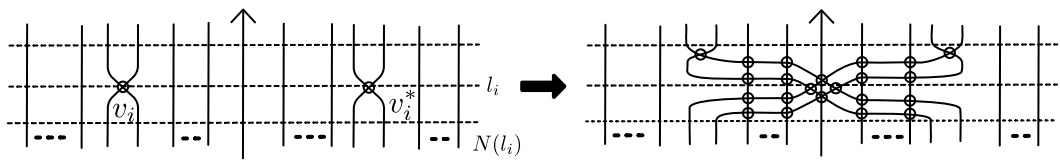}
\end{center}
\caption{2-cyclic covering diagram}\label{fig:defwcov2}
\end{figure}

For example, for the diagram $D$ depicted in Figure~\ref{fig:defwcov1} (i), the diagram $D \cup D^*$ is as in (ii). Then we have a 2-cyclic covering diagram $\widetilde{D}$ as in (iii). 

\begin{figure}[ht]
\begin{center}
\includegraphics[width=10.cm]{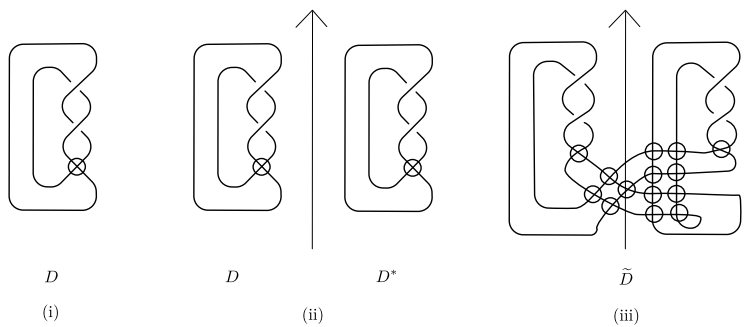}
\end{center}
\caption{A 2-cyclic covering diagram}\label{fig:defwcov1}
\end{figure}

\begin{thm}[\cite{rkn3,rkn4}]\label{thm:double}
Let $D$ and $D'$ be diagrams. If $D$ is equivalent to $D'$ as a virtual link, $\widetilde{D}$ is equivalent to $\widetilde{D'}$ as a virtual link. 
\end{thm}

Refer to \cite{rkn3,rkn4} for details.  By this theorem, the 2-cyclic covering is defined for a virtual link.  
For a virtual link $L$, the 2-cyclic covering of $L$, denoted by $\widetilde{L}$, is defined to be the equivalence class of $\widetilde{D}$ for a diagram $D$ of $L$.  

When $L$ is partitioned into $L_1, \dots, L_n$, then the 2-cyclic covering $\widetilde{L}$ is partitioned into 
$\widetilde{L}_1, \dots, \widetilde{L}_n$.

\begin{proof}[Proof of Theorem~\ref{main2}] 
Let $D_i$ $(i=1,2)$ be a loop with no crossings.  Let $D' =D'_1 \cup D'_2$ be the diagram depicted in Figure~\ref{fig:countex1} (i), and let $L$ be a virtual link presented by $D$ which is partitioned into $L_1$ and $L_2$ with $D'(L_i) = D'_i$ $(i=1,2)$.  
We assert that there is no diagram $D$ of $L$ such that the restriction $D(L_i)$ to $L_i$ 
 is isotopic to $D_i$ for $i=1, 2$.  
 
Suppose that there is a diagram $D$ of $L$ such that the restriction $D(L_i)$ to $L_i$ 
 is isotopic to $D_i$ for $i=1, 2$.  
Consider a 2-cyclic covering diagram $\widetilde{D}$ of $D$ and let 
$\widetilde{D} = \widetilde{D}_1 \cup \widetilde{D}_2$.  
The diagram $\widetilde{D}_i$ presents the sublink $\widetilde{L}_i$ for $i=1, 2$. 

Consider a 2-cyclic covering diagram $\widetilde{D'}$ of $D'$ and let 
$\widetilde{D'} = \widetilde{D'}_1 \cup \widetilde{D'}_2$.  
The diagram $\widetilde{D'}_i$ presents the sublink $\widetilde{L}_i$ for $i=1, 2$. 

By Theorem~\ref{thm:double}, the diagram $\widetilde{D}_i$ is equivalent to 
the diagram $\widetilde{D'}_i$ as a virtual link for $i=1, 2$.  

As seen in Figure~\ref{fig:countex1} (ii), the diagram $\widetilde{D'}_1$ is a diagram consisting of two components with  linking number $1$. (The linking number of a diagram with two components is the sum of signs of classical non-self crossings divided by $2$. It is an invariant of a virtual link with two components.)  

On the other hand, the diagram $\widetilde{D}_1$ is a diagram consisting of two components with linking number $0$. (This is seen as follows:  $D_1$ is a loop with no crossings, $D_1 \cup D^*_1$ is a pair of loops with no crossings.  
The diagram  $\widetilde{D}_1$ is obtained from $(D_1 \cup D_2) \cup (D^*_1 \cup D^*_2)$ by replacement as in 
Figure~\ref{fig:defwcov2}, there is no classical crossing on $\widetilde{D}_1$.) 

This is a contradiction.  
\end{proof}

\begin{figure}[ht]
\begin{tabular}{cc}
\includegraphics[width=3.cm]{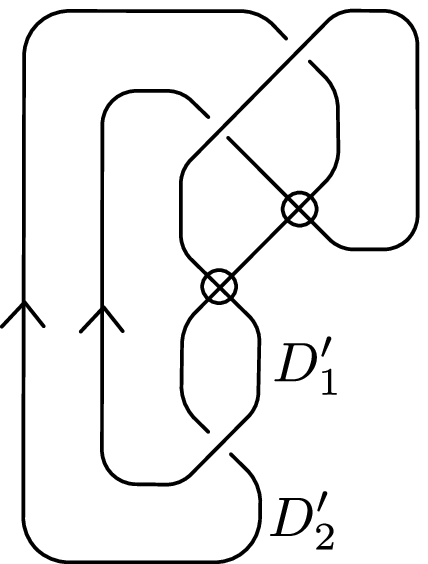}
&
\includegraphics[width=7.cm]{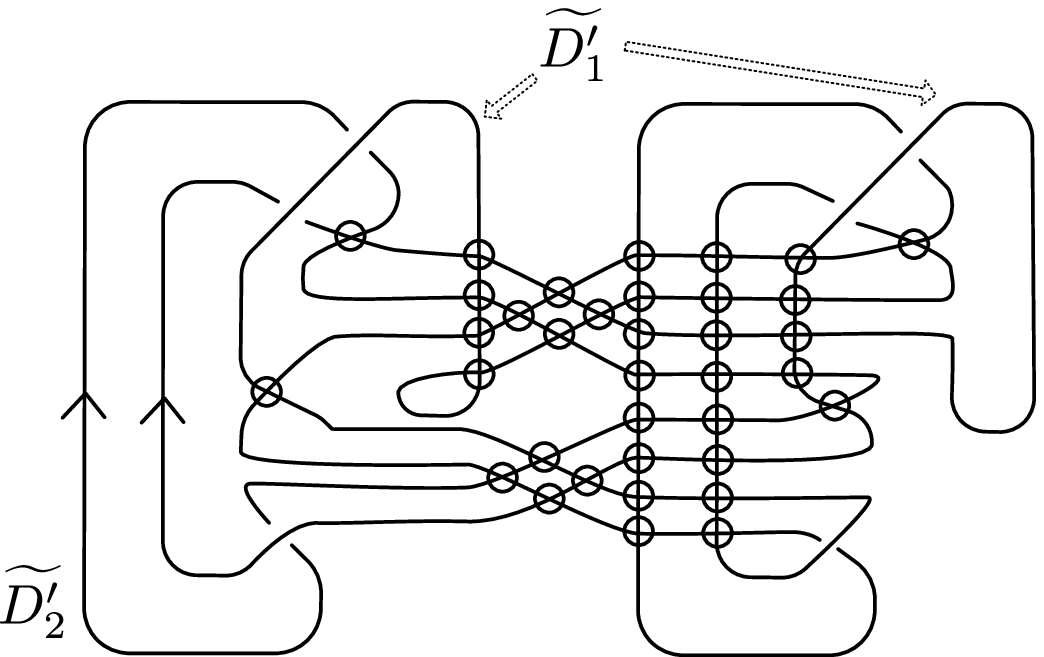}
\\
(i)&(ii)
\end{tabular}
\caption{A diagram $D'=D'_1 \cup D'_2$ and the 2-cyclic covering diagram}\label{fig:countex1}
\end{figure}

The proof above shows that there exists a virtual link $L= L_1 \cup L_2$ with two components such that when we forget $L_2$, $L_1$ is equivalent to the trivial knot  and that for any diagram $D$ of $L$, the restriction $D(L_1)$ has at least one crossing.  




\vspace{5pt}

\end{document}